\documentclass{amsart}
\usepackage[utf8]{inputenc}
\usepackage{amsmath,latexsym,amsfonts,amsthm,bm,mathtools}
\usepackage{graphicx}
\usepackage{adjustbox}
\usepackage[top=1.5in, bottom=1.5in, left=1.5in, right=1.5in]{geometry}

\newcommand{\R}{\mathbb R}
\newcommand{\E}{\mathbb E}

\newtheorem{theorem}{Theorem}
\newtheorem{corollary}[theorem]{Corollary}
\newtheorem{lemma}[theorem]{Lemma}
\newtheorem{proposition}[theorem]{Proposition}
\theoremstyle{definition}

\theoremstyle{remark}

\title{Fisher-Rao geometry of Dirichlet distributions}
\author{Alice Le Brigant}
\address{SAMM 4543, Université Paris 1 Panthéon Sorbonne, Centre PMF, Paris, France.}
\email{alice.le-brigant@univ-paris1.fr}
\author{Stephen C. Preston}
\address{Department of Mathematics, Brooklyn College and CUNY Graduate Center, New York, USA.}
\email{stephen.preston@brooklyn.cuny.edu}
\author{St\'ephane Puechmorel}
\address{Ecole Nationale de l'Aviation Civile, Université de Toulouse, Toulouse, France.}
\email{stephane.puechmorel@enac.fr}
\date{}

\begin{document}

\maketitle

\begin{abstract}
In this paper, we study the geometry induced by the Fisher-Rao metric on the parameter space of Dirichlet distributions. We show that this space is a Hadamard manifold, i.e. that it is geodesically complete and has everywhere negative sectional curvature. An important consequence for applications is that the Fréchet mean of a set of Dirichlet distributions is uniquely defined in this geometry.
\end{abstract}

\section{Introduction}

The differential geometric approach to probability theory and statistics has met increasing interest in the past years, from the theoretical point of view as well as in applications. In this approach, probability distributions are seen as elements of a differentiable manifold, on which a metric structure is defined through the choice of a Riemannian metric. Two very important ones are the Wasserstein metric, central in optimal transport, and the Fisher-Rao metric (also called Fisher information metric), essential in information geometry. Unlike optimal transport, information geometry is foremost concerned with parametric families of probability distributions, and defines a Riemannian structure on the parameter space using the Fisher information matrix \cite{fisher1922}. It was Rao who showed in 1945 \cite{rao1945} that the Fisher information could be used to locally define a scalar product on the space of parameters, and interpreted as a Riemannian metric. Later on, \u{C}encov \cite{cencov1982} proved that it was the only metric invariant with respect to sufficient statistics, for families with finite sample spaces. This result has been extended more recently to non parametric distributions with infinite support \cite{ay2015, bauer2016}.

Information geometry has been used to obtain new results in statistical inference as well as gain insight on existing ones. In parameter estimation for example, Amari \cite{amari2016information} shows that conditions for consistency and efficiency of estimators can be expressed in terms of geometric conditions; in the presence of hidden variables, the famous Expectation-Maximisation (EM) algorithm can be described in an entirely geometric manner; and in order to insure invariance to diffeomorphic change of parametrization, the so-called natural gradient \cite{amari1998natural} can be used to define accurate parameter estimation algorithms \cite{ollivier2017}.

Another important use of information geometry is for the effective comparison and analysis of families of probability distributions. The geometric tools provided by the Riemannian framework, such as the geodesics, geodesic distance and intrinsic mean, have proved useful to interpolate, compare, average or perform segmentation between objects modeled by probability densities, in applications such as signal processing \cite{arnaudon2013riemannian}, image \cite{schwander2012model, angulo2014morphological} or shape analysis \cite{peter2006, srivastava2007}, to name a few. These applications rely on the specific study of the geometries of usual parametric families of distributions, which has started in the early work of Atkinson and Mitchell. In \cite{atkinson1981}, the authors study the trivial geometries of one-parameter families of distributions, the hyperbolic geometry of the univariate normal model as well as special cases of the multivariate normal model, a work that is continued by Skovgaard in \cite{skovgaard1984}. The family of gamma distributions has been studied by Lauritzen in \cite{lauritzen1987statistical}, and more recently by Arwini and Dodson in \cite{arwini2008}, who also focus on the log-normal, log-gamma, and families of bivariate distributions. Power inverse Gaussian distributions \cite{zhang2007}, location-scale models and in particular the von Mises distribution \cite{said2019}, and the generalized gamma distributions \cite{rebbah2019} have also received attention.

In this work, we are interested in Dirichlet distributions, a family of probability densities defined on the $(n-1)$-dimensional probability simplex, that is the set of vectors of $\R^n$ with non-negative components that sum up to one. The Dirichlet distribution models a random probability distribution on a finite set of size $n$. It generalizes the beta distribution, a two-parameter probability measure on $[0,1]$ used to model random variables defined on a compact interval. Beta and Dirichlet distributions are often used in Bayesian inference as conjugate priors for several discrete probability laws \cite{o1999bayesian,griffiths2002gibbs,briggs2003probabilistic}, but also come up in a wide variety of other applications, e.g. to model percentages and proportions in genomic studies \cite{yang2017beta}, distribution of words in text documents \cite{madsen2005modeling}, or for mixture models \cite{bouguila2004unsupervised}. Up to our knowledge, the information geometry of Dirichlet distributions has not yet received much attention. In \cite{calin2014}, the authors give the expression of the Fisher-Rao metric for the family of beta distributions, but nothing is said about the geodesics or the curvature.

In this paper, we give new results and properties for the geometry of Dirichlet distributions, and its sectional curvature in particular. The derived expressions depend on the trigamma function, the second derivative of the logarithm of the gamma function, however we will avoid using its properties when possible to obtain our results. Instead, we consider a more general metric written using a function $f$, for which we only make the strictly necessary assumptions. Section \ref{sec:dirichlet} gives the setup for our problem by considering the Fisher-Rao metric on the space of parameters of Dirichlet distributions. In Section \ref{sec:general}, we consider the more general metric where $f$ replaces the trigamma function, and show that it induces the geometry of a submanifold in a flat Lorentzian space. This allows us to show geodesic completeness, and that the sectional curvature is everywhere negative. Section \ref{sec:beta} focuses on the two-dimensional case, i.e. beta distributions.

\section{Fisher-Rao metric on the manifold of Dirichlet distributions}\label{sec:dirichlet}

Let $\Delta_n$ denote the $(n-1)$-dimensional probability simplex, i.e. the set of vectors in $\R^n$ with non-negative components that sum up to one
\begin{equation*}
    \Delta_n=\{q=(q_1,\hdots,q_n)\in\R^n, \,\sum_{i=1}^nq_i=1,\, q_i\geq 0, i=1,\hdots,n\}.
\end{equation*}
The family of Dirichlet distributions is a family of probability distributions on $\Delta_n$ parametrized by $n$ positive scalars $x_1,\hdots, x_n >0$ (Figure \ref{fig:dirichlet}), that admits the following probability density function with respect to the Lebesgue measure
\begin{equation*}
f_n(q|x_1,\hdots,x_n) = \frac{\Gamma(x_1+\hdots+x_n)}{\Gamma(x_1)\hdots\Gamma(x_n)} q_1^{x_1-1}\hdots q_n^{x_n-1}.
\end{equation*}
As an open subset of $\R^n$, the space of parameters $M = (\R_+^*)^n$ is a differentiable manifold and can be equipped with a Riemannian metric defined in its matrix form by the Fisher information matrix
\begin{equation*}
g_{ij}(x_1,\hdots,x_n) = -\E\left[ \frac{\partial^2}{\partial x_i\partial x_j} \log f_n(Q|x_1,\hdots,x_n)\right], \quad i,j=1,\hdots,n,
\end{equation*}
where $\mathbb E$ denotes the expectation taken with respect to $Q$, a random variable with density $f_n(\cdot|x_1,\hdots,x_n)$. The Dirichlet distributions form an exponential family and so the Fisher-Rao metric is the hessian of the log-partition function \cite{amari2016information}, namely
\begin{equation*}
g_{ij}(x_1,\hdots,x_n) = \frac{\partial^2}{\partial x_i\partial x_j} \varphi(x_1,\hdots, x_n), \quad i,j=1,\hdots,n,
\end{equation*}
where $\varphi$ is the logarithm of the normalizing factor
\begin{equation*}
\varphi (x_1,\hdots, x_n) = \sum_{i=1}^n \log \Gamma(x_i) - \log\Gamma(x_1+\hdots+x_n).
\end{equation*}
We obtain the following metric tensor.
\begin{equation}\label{fisherraodirichlet}
g_{ij}(x_1,\hdots,x_n)= \psi'(x_i)\delta_{ij} - \psi'(x_1+\hdots+x_n),
\end{equation}
where $\delta_{ij}$ is the Kronecker delta function, and $\psi$ denotes the digamma function, that is the first derivative of the logarithm of the gamma function, i.e.
\begin{equation*}
    \psi(x) = \frac{d}{dx}\log\Gamma(x).
\end{equation*}
Its derivative $\psi'$ is called the trigamma function. As noted below, the trigamma function is a function whose reciprocal is increasing, convex, and sublinear on $\mathbb{R}^+$. For slightly greater generality, and to emphasize what properties of this function are needed for our results, we will work in the sequel with a more general function $f$ on which we make only the necessary assumptions; in our special case we have $f = 1/\psi'$.

\begin{figure}
\includegraphics[trim={1.7cm 1.3cm 1.7cm 1.3cm},clip,width=0.3\linewidth,height=10em]{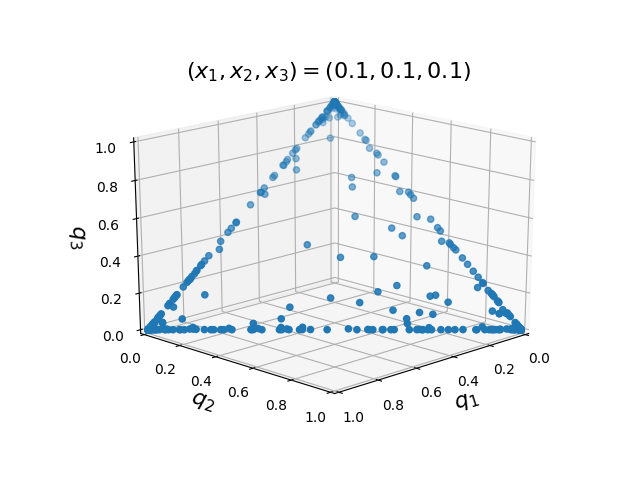}
\includegraphics[trim={1.7cm 1.3cm 1.7cm 1.3cm},clip,width=0.3\linewidth,height=10em]{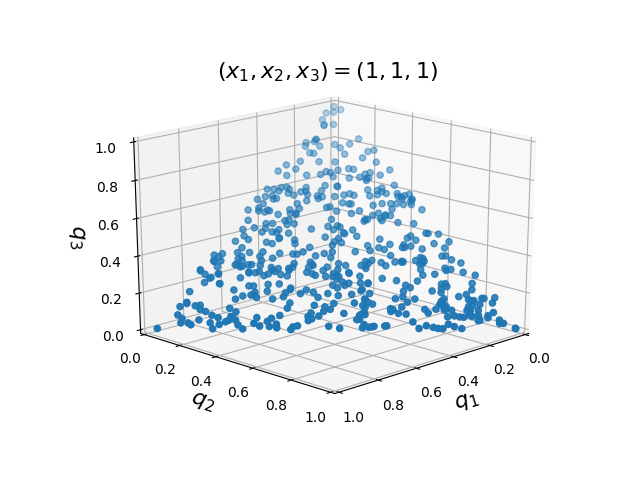}
\includegraphics[trim={1.7cm 1.3cm 1.7cm 1.3cm},clip,width=0.3\linewidth,height=10em]{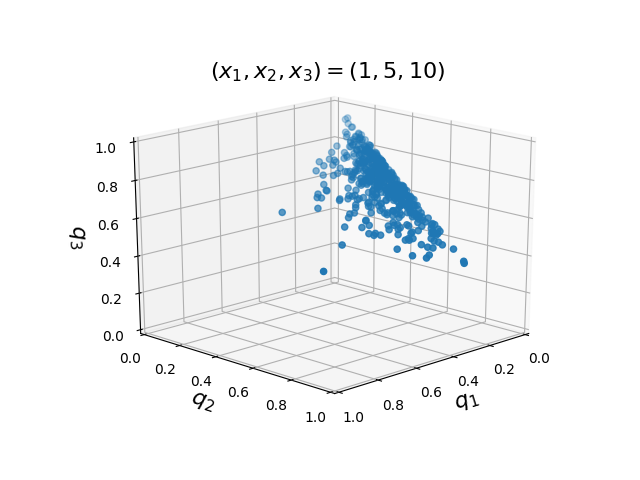}
\caption{Random samples drawn from Dirichlet distributions on the $2$-dimensional simplex $\Delta_3$ for different values of the parameters $(x_1,x_2,x_3)$.}
\label{fig:dirichlet}
\end{figure}

\section{The general framework}\label{sec:general}

\subsection{The metric}

In this section we consider a more general geometry, that admits the Fisher-Rao geometry of Dirichlet distributions as a special case. The goal is to avoid using the properties of the trigamma function when possible. For this, we consider the quadrant $M=(\R_+^*)^n$ equipped with a metric of the form
\begin{equation}\label{basicmetric}
ds^2 = \frac{dx_1^2}{f(x_1)} + \cdots + \frac{dx_n^2}{f(x_n)} - \frac{(dx_1 + \cdots + dx_n)^2}{f(x_1+\cdots + x_n)},
\end{equation}
where $f:\R_+\rightarrow \R$ is a function on which we make the following assumptions:
\begin{equation}\label{fassumptions}
f(x)\underset{x\to 0}{=}O(x^2), \quad f'(x) \underset{x\to 0}{=} O(x),\quad f(x) \underset{x\to \infty}{=} O(x^2),\quad f''>0 \text{ and } \frac{d^2}{dx^2}\left( \frac{f}{f'}\right) > 0.
\end{equation}
We retrieve the Fisher-Rao metric \eqref{fisherraodirichlet} when
\begin{equation*}
f(x)=\frac{1}{\psi'(x)}.
\end{equation*}
Notice that this choice for $f$ satisfies the conditions of \eqref{fassumptions}. Indeed, that $f(0)=f'(0)=0$ comes from the asymptotic formula $\psi'(x) \approx x^{-2}$ valid near $x=0$, since
$$ f'(0) = \lim_{x\to 0} \frac{-\psi''(x)}{\psi'(x)^2} = \lim_{x\to 0} \frac{2x^{-3}}{x^{-4}} = 0.$$
The fact that the reciprocal of the trigamma function $f(x) = \frac{1}{\psi'(x)}$ is convex comes from an argument of Trimble-Wells-Wright~\cite{trimble1989superadditive}, based on an inequality later proved in Alzer-Wells~\cite{alzer1998inequalities}. The fact that $f/f'$ is convex comes from Yang~\cite{yang2017}.

Another example of a function satisfying the conditions \eqref{fassumptions} is
\begin{equation*}\label{rationalapprox}
\tilde{f}(x) = \frac{(2x+1)x^2}{2x^2+2x+1},
\end{equation*}
a simple rational function which approximates the reciprocal of the trigamma function well, in both the small-$x$ and large-$x$ regions.

Some useful consequences of our assumptions \eqref{fassumptions} are given in the following lemma. These results are well-known, but we include the simple proofs for completeness.
\begin{lemma}\label{fconsequences}
If $f$ satisfies \eqref{fassumptions}, then we have
$f'(x)>0$ and $f(x)>0$ for all $x>0$. In addition $f$ and $f/f'$ are superadditive:
\begin{align}
f(x_1 + \cdots + x_n) &> f(x_1) + \cdots + f(x_n), \label{fsuperadditive} \\
\frac{f(x_1+\cdots + x_n)}{f'(x_1+\cdots x_n)} &> \frac{f(x_1)}{f'(x_1)} + \cdots + \frac{f(x_n)}{f'(x_n)}, \label{ffprimesuperadditive}
\end{align}
for all $x_1,\ldots, x_n>0$.
\end{lemma}

\begin{proof}
That $f''>0$ implies $f'>0$ and thus $f>0$ for all $x>0$ is obvious. It has been known since Petrovich~\cite{petrovich1932fonctionnelle} that a convex function $f$ with $f(0)=0$ is superadditive: an easy argument in the differentiable case is that
$$ f(x+y) - f(x) - f(y) = \int_0^x \int_0^y f''(s+t) \, dt \, ds \ge 0.$$
By induction the general case \eqref{fsuperadditive} follows.
Since $\lim_{x\to 0} f(x)/f'(x) = 0$, the same argument applies to $f/f'$ to give \eqref{ffprimesuperadditive}.
\end{proof}

\subsection{Lorentzian submanifold geometry}

We now show that after a change of coordinates, $M$ can be seen as a codimension 1 submanifold of the $(n+1)$-dimensional flat Minkowski space $L^{n+1}=(\R^{n+1},ds_L^2)$, where
\begin{equation}\label{minkowskimetric}
ds_L^2=dy_1^2+\hdots+dy_n^2-dy_{n+1}^2.
\end{equation}
In the sequel, we will denote by $\langle\cdot,\cdot\rangle$ the scalar product induced by this metric.

\begin{proposition}\label{flatteningprop}
The mapping
\begin{align*}
\Phi&:M\rightarrow L^{n+1},\\
&(x_1,\hdots,x_n)\mapsto (\eta(x_1),\hdots,\eta(x_n),\eta(x_1+\hdots+x_n))
\end{align*}
where $\eta\colon \mathbb{R}_+\to \mathbb{R}$ is defined by
$$ \eta(x) = \int_1^x \frac{dr}{\sqrt{f(r)}},$$
is an isometric embedding.
\end{proposition}

\begin{proof}
Since $f(x) \approx \tfrac{1}{2} f''(0) x^2$ for $x\approx 0$, we see that
$$ \int_0^1 \frac{dr}{\sqrt{f(r)}} = \infty,$$
so that the image of $\eta$ must include all negative reals. Therefore $\eta$ maps $\mathbb{R}_+$ bijectively to $(-\infty, N)$ for some $N\in (0,\infty]$. The behavior of $f$ at infinity assumed in \eqref{fassumptions} implies that $f(x) \le C x^2$ for all $x\ge K$, for some $K>0$, $C>0$, which in turns leads to
$$ \int_K^{\infty} \frac{dx}{\sqrt{f(x)}} = \infty.$$
Therefore $\eta$ maps bijectively $\mathbb{R}_+$ to $\mathbb{R}$, and $\Phi$ is a homeomorphism onto its image. Since $\eta'(x)>0$ for all $x$, it is also an immersion. Finally, if $(y_1,\hdots,y_{n+1})=\Phi(x_1,\hdots,x_n)$,
$$dy_i^2 = \eta'(x_i)^2 dx_i^2 = \frac{dx_i^2}{f(x_i)}, i=1,\hdots,n, \quad dy_{n+1}^2 = \frac{(dx_1+\hdots+dx_n)^2}{f(x_1+\hdots+x_n)},$$
and $\Phi$ is isometric.
\end{proof}

\begin{proposition}\label{positivemetric}
$S=\Phi(M)$ is a codimension $1$ submanifold of $L^{n+1}$ given by the graph of
\begin{equation}\label{submanifoldeq}
y_{n+1}=\eta(\xi(y_1)+\hdots+\xi(y_n)), \quad y_i>0,
\end{equation}
where $\xi=\eta^{-1}$. On this submanifold the metric is positive-definite and thus Riemannian. A basis of tangent vectors of $T_yS$ is defined by
\begin{equation}\label{tangentvecs}
e_i = \frac{\partial}{\partial y_i} + \sqrt{\frac{f\circ\xi(y_i)}{f\circ\xi(y_{n+1})}}\frac{\partial}{\partial y_{n+1}}, \quad i=1,\hdots,n,
\end{equation}
\end{proposition}

\begin{proof}
Let $\gamma(u)=(y_1(u),\hdots,y_{n+1}(u))$ be a parametrized curve in $S$. Then its coordinates verify the following relations
\begin{equation*}
\begin{aligned}
y_{n+1}&=\eta(\xi(y_1)+\hdots+\xi(y_n)),\\
y_{n+1}' &=\eta'(\xi(y_1)+\hdots+\xi(y_n))(\xi'(y_1)y_1'+\hdots+\xi'(y_n)y_n'),
\end{aligned}
\end{equation*}
and so, since $\xi'(x)=\sqrt{f(\xi(x))}$,
\begin{equation*}
\begin{aligned}
\gamma'(u)&=\sum_{i=1}^ny_i'(u)\frac{\partial}{\partial y_i}+\eta'(\xi(y_{n+1}(u)))(\xi'(y_1(u))y_1'(u)+\hdots+\xi'(y_n(u))y_n'(u))\frac{\partial}{\partial y_{n+1}}\\
&= \sum_{i=1}^ny_i'(u)\left(\frac{\partial}{\partial y_i} + \frac{\sqrt{f\circ\xi(y_i(u))}}{\sqrt{f\circ\xi(y_{n+1}(u))}} \frac{\partial}{\partial y_{n+1}}\right),
\end{aligned}
\end{equation*}
yielding \eqref{tangentvecs} as basis tangent vectors. 
The metric components on $S$ take the form
\begin{equation*}\label{metrictrans}
g_{ij} = \langle e_i,e_j\rangle = \delta_{ij} - W_iW_j, \qquad \text{or}\quad g = I - WW^T,
\end{equation*}
where $\langle\cdot,\cdot\rangle$ denotes the flat Minkowskian metric \eqref{minkowskimetric} and $W_i = \sqrt{f(\xi(y_i))/f(\xi(y_{n+1}))}$ for $i=1,\hdots,n$.
Applying Lemma \ref{positivetransposelemma} of the appendix gives the result upon computing
\begin{equation*}\label{Wnorm}
W^TW = \frac{\sum_{i=1}^n f\big(\xi(y_i)\big)}{f\big(\sum_{i=1}^n \xi(y_i)\big)} < 1,
\end{equation*}
by superadditivity of $f$, as in \eqref{fsuperadditive}.
\end{proof}

In other words, the metric \eqref{basicmetric} is the restriction of the flat Lorentzian metric
\begin{equation*}
ds^2 = \frac{dx_1^2}{f(x_1)} + \cdots + \frac{dx_n^2}{f(x_n)} - \frac{dt^2}{f(t)}
\end{equation*}
to the hyperplane $t = x_1 + \cdots + x_n$. In the sequel, we will use this Lorentzian submanifold geometry to study the sectional curvature and geodesic completeness of $M$. We state the results in the original coordinate system of $M$ when possible, using the following notations for any $y=(y_1,\hdots,y_{n+1})\in S$:
\begin{equation}\label{xcoord}
x_i=\xi(y_i), \,\, i=1,\hdots,n, \quad t=x_1+\hdots+x_n=\xi(y_{n+1}).
\end{equation}

\subsection{Negative sectional curvature}

The goal of this section is to prove that the sectional curvature of $M$ is everywhere negative. We start by computing the shape operator.
\begin{proposition}\label{secondff}
The shape operator of $S=\Phi(M)$ has the following components in the basis \eqref{tangentvecs} of tangent vectors
\begin{equation*}
\langle \Sigma(e_i),e_j\rangle = -\frac{1}{2\sqrt{f(t)-\sum_{\ell=1}^nf(x_\ell)}}\left( f'(x_i)\delta_{ij} - \frac{f'(t)}{f(t)}\sqrt{f(x_i)f(x_j)}\right).
\end{equation*}
\end{proposition}
\begin{proof}
We first observe that the basis vectors \eqref{tangentvecs} can be expressed in coordinates \eqref{xcoord} as
\begin{equation}\label{tangentvecsx}
e_i = \frac{\partial}{\partial y_i} + \sqrt{\frac{f(x_i)}{f(t)}}\frac{\partial}{\partial y_{n+1}}, \quad i=1,\hdots,n.
\end{equation}
Since $S$ can be obtained as the graph of $F(y_1,\hdots,y_n)=\eta(\xi(y_1),\hdots,\xi(y_n))$, a normal vector field to $S$ at $y$ is given by
\begin{equation}\label{normalfield}
N=\sum_{i=1}^n\frac{\partial F}{\partial y_i}\frac{\partial}{\partial y_i}+\frac{\partial}{\partial y_{n+1}}=\sum_{i=1}^n\sqrt{\frac{f(x_i)}{f(t)}}\frac{\partial}{\partial y_i}+\frac{\partial}{\partial y_{n+1}},
\end{equation}
which yields a timelike vector since
\begin{equation}\label{normalnorm}
\langle N,N \rangle = \frac{1}{f(t)}\left(f(x_1)+\hdots+f(x_n) - f(t)\right)<0,
\end{equation}
by superadditivity of $f$. Since $\langle N, e_i\rangle=0$, the shape operator is then given by
\begin{equation}\label{shapeop}
\langle \Sigma(e_i),e_j\rangle =-\langle \nabla_{e_i}\left(\frac{N}{\sqrt{-\langle N,N\rangle}}\right), e_j\rangle=-\frac{\langle \nabla_{e_i} N, e_j\rangle}{\sqrt{-\langle N,N \rangle}},
\end{equation}
$\nabla$ is the flat connection of the Minkowski space. Denoting $\partial_i=\partial/\partial y_i$, we get from \eqref{tangentvecsx}, \eqref{normalfield} and the flatness of $\nabla$,
\begin{equation*}
\nabla_{e_i}N = \nabla_{\partial_i }N +\frac{f(x_i)}{f(t)}\nabla_{\partial_{n+1}}N = \sum_{j=1}^n \partial_i\partial_j F \partial_j.
\end{equation*}
Inserting this last equation along with \eqref{normalnorm} into \eqref{shapeop} yields
\begin{equation*}
\langle \Sigma(e_i),e_j\rangle =-\sqrt{\frac{f(t)}{f(t)-\sum_{\ell=1}^nf(x_\ell)}}\partial_i\partial_j F.
\end{equation*}
Straightforward computations give
\begin{align*}
\partial_iF&=\eta'(t)\xi'(y_i)=\sqrt{f(x_i)/f(t)},\\
\partial_i\partial_jF &= \eta''(t)\xi'(y_i)\xi'(y_j) + \eta'(t)\xi''(y_i)\delta_{ij} = \frac{1}{2\sqrt{f(t)}}\left(\frac{-f'(t)}{f(t)}\sqrt{f(x_i)f(x_j)} + f'(x_i)\delta_{ij}\right),
\end{align*}
and the result follows after simplification.
\end{proof}

\begin{corollary}\label{posdefsecondff}
The second fundamental form given by Proposition \ref{secondff} is positive-definite.
\end{corollary}

\begin{proof}
This follows from Lemma \ref{positivetransposelemma} and the decomposition of the matrix $\Sigma$ with components $\Sigma_{ij}=\langle \Sigma(e_i),e_j\rangle$ as
\begin{equation*}
\Sigma=-\frac{1}{2}k(D-cVV^T),
\end{equation*}
where $D=\mathrm{diag}(d_1,\hdots,d_n)$ is a diagonal matrix, $V=(v_i)_{1\leq i\leq n}$ is a column vector and $c$ and $k$ are constants, defined for $i=1,\hdots,n$ by
\begin{equation}\label{decompsecondff}
d_i=f'(x_i), \quad v_i=\sqrt{f(x_i)}, \quad k = \frac{1}{\sqrt{f(t)-\sum_{\ell=1}^nf(x_\ell)}}, \quad c = \frac{f'(t)}{f(t)}.
\end{equation}
Recalling that $f>0$ and $f'>0$ by Lemma \ref{fconsequences}, we see that the matrix $D$ and constant $c$ are positive. There remains to verify that
\begin{equation*}
cV^TD^{-1}V = \frac{f'(t)}{f(t)}\sum_{i=1}^n\frac{f(x_i)}{f'(x_i)}<1,
\end{equation*}
by the superadditivity property \eqref{ffprimesuperadditive}.
\end{proof}

We can now show our main result.
\begin{theorem}\label{negseccurvthm}
The sectional curvature of the Riemannian metric \eqref{basicmetric} is negative on $M$.
\end{theorem}

\begin{proof}
We use a result from O'Neill~\cite[Chapter 4, Corollary 20]{o1983semi}, which states that if the normal vector field $N$ of a hypersurface $M$ in a flat Lorentzian manifold $L$ is timelike, then the sectional curvature of the submanifold is given by
\begin{equation}\label{gausscodazzi}
K(U,V) = - \frac{\langle \Sigma(U), U\rangle \langle \Sigma(V), V\rangle - \langle \Sigma(U),V\rangle^2}{\langle U,U\rangle \langle V,V\rangle - \langle U,V\rangle^2},
\end{equation}
where $U$ and $V$ are tangent to the submanifold and $\Sigma$ is the shape operator. The result now follows by the Cauchy-Schwarz inequality: since $\Sigma$ is a positive-definite symmetric matrix, we know that $\langle \Sigma(U),U\rangle \langle \Sigma(V),V\rangle \ge \langle \Sigma(U),V\rangle^2$ with equality iff $V$ is a multiple of $U$, but in that case the denominator vanishes as well. So the sectional curvature must be strictly negative.
\end{proof}

We now give more specifically the formula of the sectional curvature of the planes generated by the basis tangent vectors \eqref{tangentvecs}.
\begin{proposition}\label{seccurv}
The sectional curvature along the axes defined by \eqref{tangentvecs} is given by
\begin{equation*}
K(e_i,e_j)=\frac{f(x_i)f'(x_j)f'(t)+f'(x_i)f(x_j)f'(t)-f'(x_i)f'(x_j)f(t)}{4(f(t)-\sum_{\ell=1}^nf(x_\ell))(f(t)-f(x_i)-f(x_j))}.
\end{equation*}
\end{proposition}
\begin{proof}
This follows from applying formula \eqref{gausscodazzi} for the sectional curvature of a hypersurface in a flat Lorentzian manifold, with
\begin{equation*}
\langle e_i,e_i\rangle=1-\frac{f(x_i)}{f(t)}, \quad \langle e_i, e_j\rangle=\sqrt{\frac{f(x_i)f(x_j)}{f(t)^2}}, \quad i\neq j,
\end{equation*}
and $\langle \Sigma(e_i),e_j\rangle$ given by Proposition \ref{secondff}.
\end{proof}

Finally, we state a result about the eigenvalues of the shape operator, which will be useful to show geodesic completeness in the next section.
\begin{proposition}\label{boundedprincurv}
The principal curvatures at any point in $S=\Phi(M)$ are bounded.
\end{proposition}
\begin{proof}
The principal curvatures at a given point $(y_1,\hdots,y_{n+1})=\Phi(x_1,\hdots,x_n)\in S$ are the eigenvalues of the shape operator
\begin{equation*}
\Sigma=-\frac{1}{2}k(D-cVV^T),
\end{equation*}
where $D$, $V$, $c$ and $k$ are defined by \eqref{decompsecondff}. Without loss of generality, we assume that the $n$-tuple $(x_1,\hdots,x_n)$ is ordered. Let us first show that when at least $n-1$ variables go to zero, i.e. $x_i\to 0$ for $i=1,\hdots,n-1$ with the previous assumption, the principal curvatures go to zero. Let $\tau=x_1+\hdots+x_{n-1}$, then $t=x_n+\tau$ and
\begin{equation*}
f(t)-f(x_1)-\hdots -f(x_n)\underset{\tau\to 0}{\sim} \tau f'(x_n)
\end{equation*}
since $f$ has limit zero in zero, and so
\begin{equation*}
k \underset{\tau\to 0}{\sim} \frac{1}{\sqrt{\tau f'(x_n)}}.
\end{equation*}
Using the fact that when $\tau\to 0$, recalling assumptions \eqref{fassumptions},
\begin{equation*}
f(x_i)=O(x_i^2), \quad f'(x_i) = O(x_i), \quad x_i = O(\tau), \quad  i=1,\hdots,n-1,
\end{equation*}
we see that the diagonal terms of $D-cVV^T$ behave as
\begin{align*}
&f'(x_i) -\frac{f'(t)}{f(t)}f(x_i) = O(\tau), \quad i=1,\hdots,n-1,\\
&f'(x_n) -\frac{f'(t)}{f(t)}f(x_n) =\frac{f'(x_n)f(x_n+\tau) -f'(x_n+\tau)f(x_n)}{f(x_n+\tau)} \underset{\tau\to 0}{\sim} \tau\left(f''(x_n) + \frac{f'(x_n)^2}{f(x_n)}\right),
\end{align*}
while the antidiagonal terms verify
\begin{align*}
&-\frac{f'(t)}{f(t)}\sqrt{f(x_i)f(x_j)} = O(\tau^2), \quad 1\leq i,j\leq n-1,\\
&-\frac{f'(t)}{f(t)}\sqrt{f(x_i)f(x_n)} = O(\tau).
\end{align*}
Finally, we obtain that 
\begin{equation*}
\Sigma_{ij}=\langle \Sigma(e_i),e_j\rangle \underset{\tau\to 0}{=} O(\sqrt{\tau}), \quad 1\leq i,j \leq n,
\end{equation*}
and so the principal curvatures go to zero when $\tau\to 0$. Therefore there exists $\delta>0$ such that, at any point $(x_1,\hdots,x_n)$ belonging to the set
\begin{equation*}
\mathcal D_\delta=\{(x_1,\hdots,x_n)\in (\R_+^*)^n, x_i<\delta \text{ for at least } n-1 \text{ indices }i\in\{1,\hdots,n\}\},
\end{equation*}
the principal curvatures are upper bounded by, say, $1$. Now let us consider an $n$-tuple $(x_1,\hdots,x_n)\notin\mathcal D_\delta$, ordered as before. Then the diagonal elements of $D$ are ordered as well since $f'$ is increasing, and the ordered eigenvalues of $k(D-cVV^T)$ verify
\begin{equation*}
0\leq \lambda_1 \leq \hdots \leq \lambda_n \leq kd_{n},
\end{equation*}
where the lower bound comes from the positive-definiteness shown in Corollary \ref{posdefsecondff}, and the upper bound comes from \cite{golub1973some}.  Since $d_n=f'(x_n)$ and $f'$ is increasing and upper bounded by $\underset{x\to\infty}{\lim}f'(x)=1$, we have that $d_n\leq 1$. Since the function
\begin{equation*}
(x_1,\hdots,x_n)\mapsto f(x_1+\hdots+x_n)-f(x_1)-\hdots-f(x_n)
\end{equation*}
is increasing in all of its variables, it is larger than its limit as the first $n-2$ variables go to zero, and since at least $x_{n-1}>\delta$ and $x_n>\delta$, we obtain
\begin{equation*}
k =\frac{1}{\sqrt{f(t)-f(x_1)-\hdots -f(x_n)}}\leq \frac{1}{\sqrt{f(2\delta)-2f(\delta)}},
\end{equation*}
and the principal curvatures are again bounded.
\end{proof}

\subsection{Geodesics and geodesic completeness}

The geodesics of $M$ for the metric \eqref{basicmetric} are parametrized curves $u\mapsto(x_1(u),\hdots,x_n(u))$ solution of the standard second-order ODEs
\begin{equation*}
\ddot x_k + \sum_{1\leq i,j\leq n}\Gamma_{ij}^k\dot x_i\dot x_j=0, \quad k=1,\hdots,n,
\end{equation*}
whose coefficients can be computed using the following result.
\begin{proposition}\label{prop:christoffel}
The Christoffel symbols for metric \eqref{basicmetric} are given by
\begin{align*}
    \Gamma_{ij}^k = \frac{1}{2}\left[\frac{f(x_k)}{f(t)-\sum_{\ell=1}^n f(x_\ell)}\left(g(t)-g(x_j)\delta_{ij}\right) - g(x_k)\delta_{ij}\delta_{jk}\right],
\end{align*}
where $t=x_1+\hdots+x_n$ and $g(x) = f'(x)/f(x)$, while $\delta$ denotes the Kronecker delta function.
\end{proposition}
\begin{proof}
The Christoffel symbols of the second kind $\Gamma_{ij}^k$ can be obtained from the Christoffel symbols of the first kind $\Gamma_{ijk}$ and the coefficients $g^{ij}$ of the inverse of the metric matrix using the formula
\begin{equation*}
\Gamma_{ij}^k = \Gamma_{ijl} g^{kl},
\end{equation*}
where we have used the Einstein summation convention. It is easy to see that the Christoffel symbols of the first kind are given by
\begin{equation*}
\Gamma_{ijk} = \frac{1}{2}\left(\frac{f'(t)}{f(t)^2}-\frac{f'(x_k)}{f(x_k)^2}\delta_{ik}\delta_{jk}\right).
\end{equation*}
Applying the Sherman-Morrison formula, we obtain that the inverse of the metric matrix
\begin{equation*}
g(x_1,\hdots,x_n) = \text{diag}\left(\frac{1}{f(x_1)},\hdots,\frac{1}{f(x_n)}\right) - \frac{1}{f(t)} J,
\end{equation*}
where $J$ denotes the $n$-by-$n$ matrix with all entries equal to one, is given by
\begin{equation}\label{invmetric}
g(x_1,\hdots,x_n)^{-1}= \text{diag}(f(x_1),\hdots,f(x_n)^{-1}) + \frac{1}{f(t)- \sum_{\ell=1}^nf(x_\ell)}[f(x_i)f(x_j)]_{1\leq i,j\leq n}.
\end{equation}
Noticing that the sum of all the elements of the $k$th line (or column) of the inverse of the metric matrix is given by
\begin{equation}\label{sumline}
\sum_{\ell=1}^ng^{k\ell}= f(x_\ell)+\frac{\sum_{\ell=1}^nf(x_k)f(x_\ell)}{f(t)- \sum_{\ell=1}^nf(x_\ell)}=\frac{f(x_k)f(t)}{f(t)- \sum_{\ell=1}^nf(x_\ell)},
\end{equation}
we obtain
\begin{equation*}
\Gamma_{ij}^k = \frac{1}{2}\sum_{\ell=1}^n\left(\frac{f'(t)}{f(t)^2}-\frac{f'(x_\ell)}{f(x_\ell)^2}\delta_{ij}\delta_{j\ell}\right)g^{k\ell}=\frac{1}{2}\frac{f'(t)}{f(t)^2}\sum_{\ell=1}^ng^{k\ell}-\frac{1}{2}\delta_{ij}\frac{f'(x_j)}{f(x_j)^2}g^{kj}.
\end{equation*}
Inserting \eqref{sumline} and the general term of the inverse matrix \eqref{invmetric} in the above yields
\begin{equation*}
\Gamma_{ij}^k=\frac{1}{2}\frac{f'(t)}{f(t)^2}\frac{f(x_k)f(t)}{f(t)- \sum_{\ell=1}^nf(x_\ell)} -\frac{1}{2}\delta_{ij}\frac{f'(x_j)}{f(x_j)^2}\left(f(x_k)\delta_{kj} + \frac{f(x_k)f(x_j)}{f(t)-\sum_{\ell=1}^nf(x_\ell)}\right)
\end{equation*}
and the result follows.
\end{proof}

Now, using the result of Proposition \ref{boundedprincurv} and a theorem from \cite{harris1988closed}, we can show that $M$ is geodesically complete.
\begin{theorem}\label{geodcompletethm}
$M$ equipped with the Riemannian metric \eqref{basicmetric} is geodesically complete.
\end{theorem}
\begin{proof}
The image of $M$ by $\Phi$ is a hypersurface of the $(n+1)$-Minkowski space $L^{n+1}$. Moreover, $\Phi$ is an embedding and it is closed since $\Phi(M)$ is a closed subset of $L^{n+1}$ as preimage of the singleton $\{0\}$ by the continuous map $(y_1,\hdots,y_{n+1})\mapsto\xi(y_1)+\hdots+\xi(y_n)-\xi(y_{n+1})$. Therefore $\Phi$ is proper \cite[Theorem 1]{harris1988closed}. Then, \cite[Theorem 6]{harris1988closed} allows us to conclude that since $\Phi$ has bounded principal curvatures by Proposition \ref{boundedprincurv}, $M$ equipped with the pullback \eqref{basicmetric} of the Minkowski metric by $\Phi$ is complete.
\end{proof}

\subsection{Uniqueness of the Fréchet mean} 

Since $M$ is simply connected, we deduce from Theorems \ref{negseccurvthm} and \ref{geodcompletethm} the following.
\begin{corollary}
$M$ equipped with the Riemannian metric \eqref{basicmetric} is a Hadamard manifold.
\end{corollary}
This has important implications in information geometry, as it guarantees the uniqueness of the Fréchet mean of a set of points in this geometry. The Fréchet mean, also called intrinsic mean, is a popular choice to extend the notion of barycenter to a Riemannian manifold. It is defined for a set of points $p_1,\hdots, p_N\in M$ as the minimizer of the sum of the squared geodesic distances to the points of the set
\begin{equation*}
\bar p = \underset{p\in M}{\mathrm{argmin}}\sum_{i=1}^Nd(p, p_i)^2.
\end{equation*}
It exists as long as $M$ is complete, however it is in general not unique and refers to a set. Uniqueness holds however for Hadamard manifolds \cite{karcher1977}. This implies that the notion of barycenter of Dirichlet distributions is well defined in the Fisher-Rao geometry.

\section{The two-dimensional case of beta distributions}\label{sec:beta}

The simplest case is obviously when $n=2$, and even in this case the formulas are nontrivial. When $f=\frac{1}{\psi'}$, the metric comes from the
well-known two-parametric family of beta distributions defined on the compact interval $[0,1]$, which is important in statistics and useful in many applications.

\begin{proposition}
The geodesic equations are given by
\begin{equation}\label{geodeq}
\begin{aligned}
a(x,y)\ddot x + b(x,y) \dot x^2 + c(x,y)\dot x\dot y + d(x,y)\dot y^2 = 0,\\
a(y,x)\ddot y + b(y,x) \dot y^2 + c(y,x)\dot x\dot y + d(y,x)\dot x^2 = 0,
\end{aligned}
\end{equation}
where
\begin{align*}
a(x,y) &= 2\big[ f(x+y) - f(x)-f(y)\big] \\
b(x,y) &= f(y) g(x) + f(x) g(x+y) - f(x+y)g(x)  \\
c(x,y) &= 2f(x) g(x+y) \\
d(x,y) &= f(x)g(x+y) - g(y)f(x),
\end{align*}
with the shorthand $g(x) = f'(x)/f(x)$.
\end{proposition}
\begin{proof}
The geodesic equations can be expressed in terms of the Christoffel symbols as
\begin{align*}
\ddot x + \Gamma_{11}^1 \dot x^2 + 2\Gamma_{12}^1\dot x\dot y + \Gamma_{22}^1\dot y^2 = 0,\\
\ddot y + \Gamma_{22}^2 \dot y^2 + 2\Gamma_{12}^2\dot x\dot y + \Gamma_{11}^2\dot x^2 = 0,
\end{align*}
and the coefficients can be computed using Proposition \ref{prop:christoffel}. 
\end{proof}

\begin{figure}
\centering
\includegraphics[width=13em, height=12em]{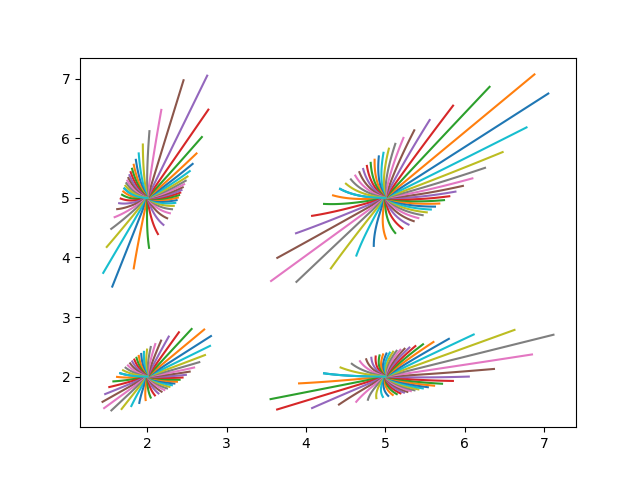}
\includegraphics[width=14em, height=12em]{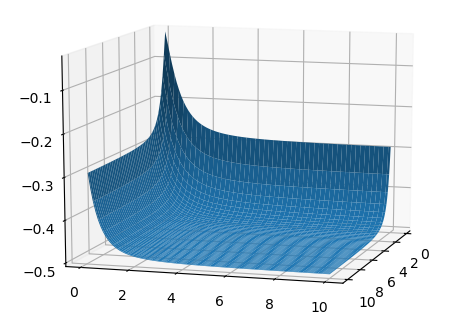}
\caption{On the left, geodesic balls and on the right, sectional curvature of the manifold of beta distributions ($n=2$).}
\label{fig:dim2}
\end{figure}

\begin{figure}
\centering
\includegraphics[width=0.35\textwidth]{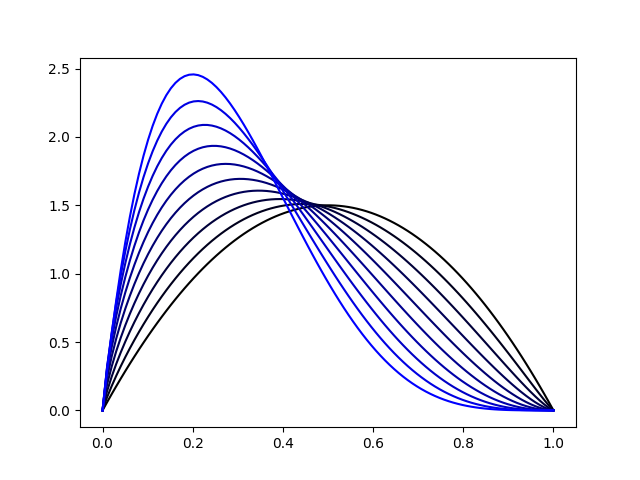}
\includegraphics[width=0.35\textwidth]{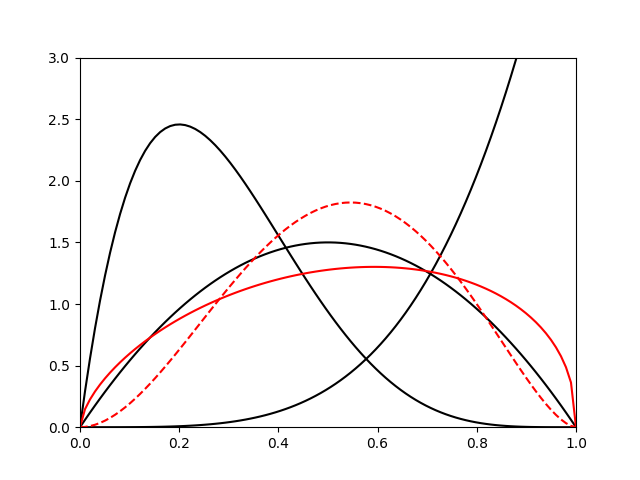}
\caption{On the left, geodesic between the beta distributions of parameters $(2,5)$ and $(2,2)$ and on the right, Fréchet mean (full red line) compared to the Euclidean mean (dashed red line) of the beta distributions of parameters $(2,5)$, $(2,2)$ and $(5,1)$, shown in terms of probability density function.}
\label{fig:dim2dens}
\end{figure}

No closed form is known for the geodesics, but they can be computed  numerically by solving \eqref{geodeq}, see the left-hand side of Figure \ref{fig:dim2}. Nonetheless we can notice that, due to the symmetry of the metric with respect to parameters $x$ and $y$, both equations in \eqref{geodeq} yield a unique ordinary differential equation when $x=y$.

\begin{corollary}
Solutions of the geodesic equation \eqref{geodeq} with $x(0)=y(0)$ and $\dot{x}(0)=\dot{y}(0)$ satisfy 
\begin{equation}\label{diagconserv}
 \sqrt{q(x(t))} \dot{x}(t) = \text{constant},
 \end{equation}
where $q(x) = \frac{1}{f(x)} - \frac{2}{f(2x)}$, 
and thus can be found by quadratures. 
\end{corollary}

\begin{proof}
If at some time $t_0$ we have $x=y$ and $\dot{x}=\dot{y}$, then the equations \eqref{geodeq} imply that $\ddot{x}=\ddot{y}$ at $t_0$. Differentiating repeatedly in time shows that all higher derivatives must also be equal at $t_0$, and we conclude by analyticity of the solutions that $x(t)=y(t)$ on some interval. The usual extension arguments for ODEs then imply that $x(t)=y(t)$ on the entire domain of the solution, which by Theorem \ref{geodcompletethm} is $\mathbb{R}$.

When $x=y$ equation \eqref{geodeq} reduces to  
$$ 2 \big[ f(2x) - 2f(x)\big] \ddot{x} + \left(\frac{4f(x)f'(2x)}{f(2x)} - \frac{f(2x) f'(x)}{f(x)} \right) \dot{x}^2 = 0,$$
which is equivalent to  
$$ 2 q(x) \ddot{x} + q'(x) \dot{x}^2 = 0.$$
This clearly implies the conservation law \eqref{diagconserv}. 
The differential equation \eqref{diagconserv} can then be solved by writing 
$$t = \frac{1}{\dot{x}_0 \sqrt{q(x_0)}} \, \int_{x_0}^x \sqrt{q(s)}\, ds$$
and inverting the resulting function.
\end{proof} 

For example, if $f(x)=1/\psi'(x)$, then the duplication formula for the trigamma function implies 
$$ q(x) = \psi'(x) - 2\psi'(2x) = \tfrac{1}{2} [\psi'(x) - \psi'(x+\tfrac{1}{2})].$$
Asymptotically this looks like $q(x)\approx \frac{1}{2x^2}$ for $x\approx 0$ and $q(x)\approx \frac{1}{4x^2}$ as $x\to \infty$. 
We conclude that it takes infinite time for a geodesic along the diagonal to either reach ``diagonal infinity'' or the origin, 
as Theorem \ref{geodcompletethm} of course implies. 

From an applications point of view, the geodesics for the Fisher-Rao geometry allow us to define a notion of optimal interpolation between beta and more generally Dirichlet distributions. An example of such an optimal interpolation is shown on the left-hand side of Figure \ref{fig:dim2dens}$,$ in terms of probability density function.

Now we give the formula for the sectional curvature in two dimensions.

\begin{proposition}\label{gaussiancurvprop}
If $n=2$, the sectional curvature is given by
\begin{equation}\label{gaussian}
K(x,y) = -\frac{1}{4} \, \frac{f(t)f'(x)f'(y) - f(x)f'(t)f'(y)-f(y)f'(t)f'(x)}{\big[ f(t)-f(x)-f(y)\big]^2},
\end{equation}
where $t=x+y$.
\end{proposition}
\begin{proof}
This is just a particular case of Proposition \ref{seccurv}.
\end{proof}

Notice that in two dimensions, the negativity of the sectional curvature is straightforward, as there is only one Gaussian curvature to consider, which is given by \eqref{gaussian}, in which one can easily see that the numerator is positive by factorizing by $f'(x)f'(y)f'(t)>0$ and using the superadditivity property \eqref{ffprimesuperadditive} of $f/f'$.

As previously mentioned, the negative curvature of the Fisher-Rao geometry also has interesting implications for applications: it entails that the Fréchet mean of a set of beta, or more generally Dirichlet distributions is well defined. An example of Fréchet mean of beta distributions is shown in terms of probability density function on the right-hand side of Figure \ref{fig:dim2dens}.

Numerically we observe that when $f=1/\psi'$, the function $K(x,y)$ given by \eqref{gaussian} is decreasing in both the $x$ and $y$ variables -- see the right-hand side of Figure \ref{fig:dim2} -- but we do not yet have a proof of this fact. However we may analyze the asymptotics of the function relatively easily.

\begin{figure}
    \centering
    \includegraphics[width=0.3\linewidth]{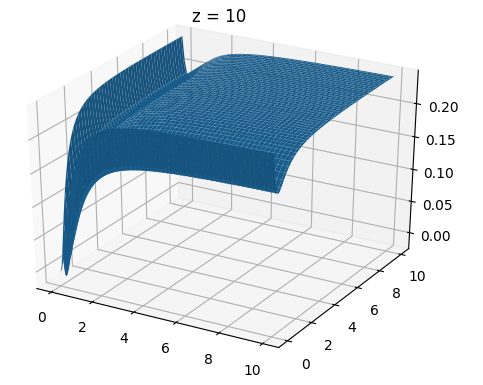}
    \includegraphics[width=0.3\linewidth]{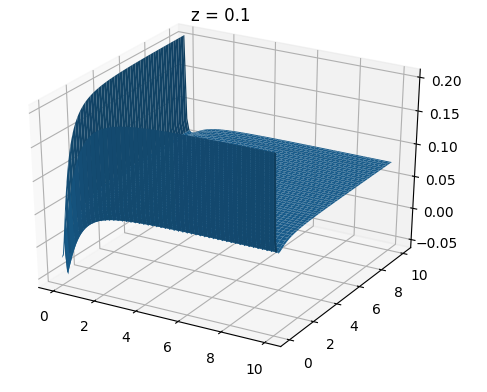}
    \includegraphics[width=0.3\linewidth]{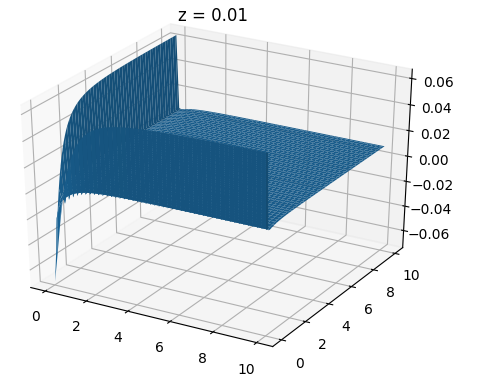}
    \caption{The difference \eqref{diffseccurv} between the sectional curvatures of the plane generated by $e_1$ and $e_2$ in two and three dimensions changes sign for $z=0.01$.}
    \label{fig:dim2dim3}
\end{figure}

\begin{proposition}
\label{prop:asymptotic}
If $f=1/\psi'$, then the asymptotic behavior of the sectional curvature given by \eqref{gaussian} approaching the boundary square is given by
\begin{align}
\lim_{y\to 0}K(x, y) &= \lim_{y\to 0}K(y, x) = \frac{3}{4} - \frac{\psi'(x)\psi'''(x)}{2\, \psi''(x)^2}, \label{zeroasymptote} \\
\lim_{y\to \infty} K(x, y) &= \lim_{y\to \infty} K(y, x) = \frac{x\, \psi''(x) + \psi'(x)}{4(x\,\psi'(x) - 1)^2}.\label{infiniteasymptote}
\end{align}
Moreover, we have the following limits at the asymptotic corners:
\begin{align*}
\lim_{x,y\to 0} K(x,y) = 0, \quad \lim_{x,y\to \infty}K(x,y) = -\frac{1}{2},\quad
\lim_{x\to 0, y\to \infty} K(x, y)=\lim_{x\to \infty, y\to 0} K(x, y) = -\frac{1}{4}.
\end{align*}
\end{proposition}

\begin{proof}
Writing $K(x,y) = -\frac{A(x,y)}{4B(x,y)^2}$, with 
\begin{align*}
A(x,y) &= f(x+y)f'(x)f'(y) - f(x)f'(x+y)f'(y) - f(y)f'(x+y)f'(x), \\
B(x,y) &= f(x+y)-f(x)-f(y),
\end{align*}
we note that $A(x,0)=M_y(x,0)=0$ and $N(x,0)=0$, so that $$\displaystyle \lim_{y\to 0} K(x,y) = \frac{A_{yy}(x,0)}{8B_y(x,0)^2},$$
which gives \eqref{zeroasymptote} after rewriting in terms of $\psi'$. 

For the infinite limits, we use the facts that $\displaystyle \lim_{y\to\infty} f(y)-y = -\tfrac{1}{2}$ and $\lim_{y\to\infty} f'(y)=1$, and that $\lim_{y\to\infty} y(f'(y)-1) = 0$, to obtain limits of $A(x,y)$ and $B(x,y)$ separately with elementary computations.
\end{proof}

These limits and strong numerical evidence allow us to conjecture that the sectional curvature in two dimensions is lower bounded by $-1/2$.
Comparing the two-dimensional sectional curvature $K_2(x,y)=K(x,y)$ with the sectional curvature of the plane generated by $e_1$ and $e_2$ in three dimensions, that we denote by $K_3(x,y,z)$, we observe numerically that for a given $z>0$, the function
\begin{equation}\label{diffseccurv}
    (x,y)\mapsto K_3(x,y,z)-K_2(x,y)
\end{equation}
does not have a fixed sign in general, as can be observed on Figure \ref{fig:dim2dim3} for small values of $x$, $y$ and $z$.

\section*{Acknowledgments}

S. C. Preston was partially supported by Simons Foundation, Collaboration Grant for Mathematicians, no. 318969. A. Le Brigant and S. Puechmorel would like to thank Fabrice Gamboa and Thierry Klein for bringing this problem to their attention and for fruitful discussions.

\section*{Appendix}
Here we give a well-known principle to establish positivity of matrices.

\begin{lemma}\label{positivetransposelemma}
Suppose $A$ is a positive-definite symmetric matrix, $V$ is a vector, and $c$ is a positive real number.
Then $B = A - c VV^T$ is positive-definite if and only if
$$c V^TA^{-1}V < 1.$$
\end{lemma}

\begin{proof}
Since $A$ is positive-definite and symmetric, we may write $A = P^2$ for some positive-definite symmetric matrix $P$.
Let $X = P^{-1}V$; then we may write
$$ B = P^2 - cVV^T = P\big(I - c(P^{-1}V) (P^{-1}V)^T \big)P = P (I - cXX^T) P.$$
Denoting by $\langle U|U \rangle=U^TU$ the usual scalar product on $\R^n$, we have for any vector $U$,
\begin{align*}
\langle U|BU\rangle &= \langle PU|PU\rangle - c \langle PU| X\rangle^2 = \lvert Y\rvert^2 - c\langle Y|X\rangle^2 \\
&\ge \lvert Y\rvert^2 - c\lvert X\rvert^2 \lvert Y\rvert^2 = \lvert Y\rvert^2 (1 - c\lvert X\rvert^2),
\end{align*}
where $Y = PU$, using the Cauchy-Schwarz inequality. This is positive for all $U$ if and only if the right side is positive for all $Y$,
which translates into $c \lvert X\rvert^2 < 1$. Since $\lvert X\rvert^2 = \langle P^{-1}V | P^{-1}V\rangle = \langle V| A^{-1}V\rangle$,
we obtain the claimed result.
\end{proof}

\bibliographystyle{plain}
\bibliography{bibliography}

\end{document}